\documentclass[a4paper,reqno]{amsart}
\usepackage{amsmath,amsfonts,amssymb,enumerate,amsthm,graphics,color,graphicx,datetime}
\usepackage{newtxtext,newtxmath}  
\usepackage{dsfont}

\usepackage[colorlinks=true]{hyperref} 

\theoremstyle{plain}
\newtheorem{thm}{Theorem}[section]
\newtheorem{lem}[thm]{Lemma}
\newtheorem{prop}[thm]{Proposition}

\theoremstyle{definition}

\theoremstyle{remark}

\def\Rde{{\mathbb{R}^d}}
\def\Ne{{\mathbb{N}}}
\def\Re{{\mathbb R}}
\def\dd{\text{\rm\,d}}

\def\al{\alpha}

\def\be{\beta}

\def\sg{\sigma}
\def\f{\varphi}

\def\Om{\Omega}
\def\om{\omega}

\def\spt{\operatorname{supp}}  

\def\sT{\,;\;}

 \numberwithin{equation}{section}
\begin{document}

\title{Existence of an extremal function
of Sobolev critical embedding with an $\al$-homogeneous weight}

\renewcommand*{\urladdrname}{\itshape ORCID}

\author[P. Gurka]{Petr Gurka}
\address{P. Gurka, Department of Mathematics,
Czech University of Life Sciences Prague, 165 21, Prague 6, Czech Republic;
Department of Mathematics, College of Polytechnics Jihlava, Tolstého 16, 586 01, Jihlava, Czech Republic;
Department of Mathematical Analysis, Faculty of Mathematics and Physics, Charles University, Sokolovská 83, 186 75 Praha 8, Czech Republic}
\email{gurka@tf.czu.cz}
\urladdr{\href{https://orcid.org/0000-0002-0995-4711}{0000-0002-0995-4711}}

\thanks{The research of the first author was supported by grant No.~GA23-04720S
 of the Czech Science Foundation; the second author's research was partially supported by the two
 Australian Research Council grant DP200101065 and DP220100067.}

\author[D. Hauer]{Daniel Hauer}
\address{D. Hauer, Brandenburg University of Technology Cottbus–Senftenberg, Platz der Deutschen Einheit 1,
  03046 Cottbus, Germany;
  School of Mathematics and Statistics, The
  University of Sydney, Sydney, NSW, 2006, Australia}
\email{daniel.hauer@b-tu.de}
\urladdr{\href{https://orcid.org/0000-0001-6210-7389}{0000-0001-6210-7389}}



\begin{abstract}
In our previous publication [{\em Calc. Var. Partial Differential Equations}, 60(1):Paper No. 16,
  27, 2021], we delved into examining a critical
Sobolev-type embedding of a Sobolev weighted space into an exponential
weighted Orlicz space. We specifically determined the optimal Moser-type
constant for this embedding, utilizing the monomial weight introduced
by Cabr\'e and Ros-Oton [{\em J. Differential Equations}, 255(11):4312--4336, 2013].
Towards the conclusion of that paper,
we pledged to explore the existence of an extremal function within this framework.
In this current work, we not only provide a positive affirmation to this inquiry
but extend it to a broader range of weights known as \emph{$\al$-homogeneous weights}.
\end{abstract}

\maketitle

\section{Introduction}

It is well-known that Trudinger-Moser type inequalities can be considered as the critical case
of Sobolev embeddings. To be more precise, recall that the Sobolev space
$W^{1,d}(\Rde)$, $d\in\Ne$,
is embedded into the Lebesgue space
$L^q(\Rde)$ for any $q\in[d,\infty)$, although
$W^{1,d}(\Rde)$ is not embedded into $L^\infty(\Rde)$.
It was proved in 1967 by Trudinger~\cite{Tr} that there is an embedding of this type of
$W^{1,d}_0(\Om)$ into the Orlicz space $L^\Phi(\Om)$, with $\Om$
being a bounded domain in $\Rde$,
and the Young function $\Phi(x)=\exp\big(\alpha|x|^{d/(d-1)}\big)$, $t>0$,
where $\al$ is a positive constant. However, this embedding was announced
earlier by Yudovich~\cite{Yu} without a proof and,
in a slightly weaker form, proved by Pokhozhaev~\cite{Po}.
In 1971 Moser~\cite{Moser71} showed that if $\al\le\al_d=d\om_{d-1}^{1/(d-1)}$,
where $\om_{d-1}$ denotes the surface area of the unit sphere in $\Rde$, then
\begin{equation*}
  \sup_u\int_{\Om}\exp\big(\alpha|x|^{d/(d-1)}\big)\dd x<\infty,
\end{equation*}
where the supremum is taken over all functions $u$ from the unit ball of the space $W^{1,d}_0(\Om)$.
When $\al>\al_d$, this supremum becomes infinite.

A natural question arises as to whether there exists a function $u$  such that
the supremum is attained. The first step towards solving this problem was made in 1986
by Carleson and Chang~\cite{CaCh}, who showed that an extremal function $u$  exists
if $\Om$ is a ball in $\Rde$. In 1988, Struwe~\cite{MR970849} demonstrated that this result
continues to hold for small perturbations of the ball in $\Re^2$, and in 1992, Flucher~\cite{Flu}
extended the result to any bounded domain in $\Re^2$. Finally, in 1995, Lin~\cite{Lin}
generalised the result to any dimension.
Note that a key ingredient of his proof is the use of $d$-Green's functions,
the singular solutions to the $d$-Laplacian.

In our previous paper, published in 2021~\cite{GuHa21}, we studied a weighted version of the
Trudinger-Moser inequality and its relationship with the corresponding Sobolev-type embedding.
This research was based on recent results by Cabr\'e\ and Ros-Oton~\cite{CaRoOt13} and Lam~\cite{Lam17},
whose function spaces involved the so-called \emph{monomial} weight. In our paper,
we presented some weighted analogues of the classical results.
Namely, we found the ``optimal'' Orlicz spaces (within the class of Orlicz spaces) of the corresponding
critical Sobolev-type embedding and showed that this embedding is not compact. Moreover, we proved,
that an embedding into any properly larger Orlicz space is compact, and we also derived
a corresponding \emph{concentrated compactness} principle of P.-L. Lions.
However, we left the question of the existence of an extremal function unsolved.
In this paper, we aim to address this gap, at least for the case where $\Om$ is a ball in $\Rde$.
We are working in a slightly more general context, as our weight is a \emph{homogeneous} weight
(more details about the weights can be found in Section~\ref{Prelim}).

\section{Preliminaries} \label{Prelim}

\subsection*{Weighted measure on an convex cone}
We denote by $\Sigma$  an open convex cone in $\Rde$, $d\in\Ne$, $d\geq2$, with vertex at the origin.
Following \cite{CRoS16}, $w\colon \overline{\Sigma} \to [0, \infty)$ is a nonnegative (not identically zero)
continuous function that is \emph{$\al$-homogeneous}, $\al>0$,
and such that $w^{1/\alpha}$ is concave in $\Sigma$. Recall that $w$ is $\al$-ho\-mo\-geneous if
\begin{equation} \label{Wfunc}
  w(\kappa x)=\kappa^\al w(x)\quad\text{for any $x\in\overline{\Sigma}$ and all $\kappa>0$}.
\end{equation}
Throughout the paper $\mu$ denotes the weighted measure on $\Sigma$ defined as
\begin{equation*}
  \dd\mu(x)=w(x) \dd x.
\end{equation*}
For future reference, we set
\begin{equation}\label{NumD}
  D=d+\al.
\end{equation}
Important examples of $\al$-homogeneous functions
are the \emph{monomial weights}. Namely, let
$\Sigma = \Sigma_1 \cap \cdots \cap \Sigma_k$, $k\in\{1, \dots, d\}$,
where $\Sigma_j\neq\Rde$, $j=1, \dots, k$, are open convex cones
in $\Rde$ with vertex at the origin.
Given $A_1, \dots, A_k>0$, the
weight $w\colon \overline{\Sigma} \to [0, \infty)$ defined as
\begin{equation*}
w(x) = \prod_{j = 1}^k \operatorname{dist}(x, \partial \Sigma_j)^{A_j},\ x\in \overline{\Sigma},
\end{equation*}
satisfies the $\al$-homogeneity assumptions. For other interesting examples we refer the reader
to \cite[Section~2]{CRoS16}.
In particular, when
$\Sigma_j = \{x\in\Rde\sT x_j > 0\}$, $j=1, \dots, k$,
the function (cf. also \cite{CaRoOt13})
\begin{equation*}
w(x) = x_1^{A_1}\cdots x_k^{A_k},\ x\in \overline{\Sigma},
\end{equation*}
we considered in our previous paper \cite{GuHa21}.

\subsection*{Basic definitions}

\paragraph{\it Measure space.}\quad
By the symbol $(X,\nu)$ we denote a measure space~$X$ with a nonnegative $\sigma$-finite measure~$\nu$.
If $\Om$ is a $\nu$-measurable subset of $X$, we denote by $\nu(\Om)$  the $\nu$-measure of $\Om$, that is,
$\nu(\Om)=\int_\Om\dd\nu(x)$.
If $X=\Rde$ with  its $n$-dimensional Lebesgue measure $\dd\nu(x)=\dd x$ and
$M\subset\Rde$ is a Lebesgue measurable set in $\Rde$, we use the notation
$|M|=\int_\Om\dd x$.

\paragraph{\it Lebesgue space.}\quad
If $(X,\nu)$ is a measure space,
we denote by $L^p(\Om,\nu)$, $p\in[1,\infty]$, the \emph{Lebesgue space},
of all  $\nu$-measurable functions~$f$
on a $\nu$-measurable set $\Om$ in $X$, equipped with the norm
\begin{equation*}
  \|f\|_{p,\Om,\nu}=
  \bigg\{
    \begin{array}{ll} \vspace{3pt}
      \big(\int_\Om |f(x)|^p\dd \nu(x)\big)^{1/p} & \hbox{if $p<\infty$,} \\
      \text{{$\nu$}-ess sup\,}_{x\in\Om}|f(x)| & \hbox{if $p=\infty$.}
    \end{array}
\end{equation*}
If $X=\Rde$ and $\nu$ is the $d$-dimensional Lebesgue measure,
then we write $\|\cdot\|_{p,\Om}$ instead of $\|\cdot\|_{p,\Om,\nu}$. Moreover,
we  simply write  $\|\cdot\|_{p,\nu}$ (or $\|\cdot\|_{p}$) when $\Om=X$  (or $\Om=\Rde$).
For $p\in[1,\infty]$ we define the \emph{H\"older conjugate}
 number $p'\in[1,\infty]$ by the equality
$\frac1p+\frac1{p'}=1$.

\paragraph{\it Sobolev space.}\quad
Let $\Om$ be a domain in a measure space $(\Rde,\nu)$
with a nonnegative Borel measure~$\nu$ and let $p\in[1,\infty]$.
The \emph{Sobolev space} $W^{1,p}_0(\Om,\nu)$ is defined  as
the closure of $C_c^\infty(\Om)$ (the space of infinitely
differentiable functions with compact supports in $\Om$)
with respect to the norm
\begin{equation*}
  \|u\|=\|\nabla u\|_{p,\Om,\nu},
\end{equation*}
where $\nabla u$ is the gradient of $u$ and $|\nabla u|$
is its Euclidean length, that is,
\begin{equation*}
  \nabla u=\Big(\frac{\partial u}{\partial x_1},\dots,\frac{\partial u}{\partial x_n}\Big),
  \quad
  |\nabla u|=\Big(\sum_{j=1}^{n}({\partial u}/{\partial x_j})^2\Big)^{1/2}
\end{equation*}
(to simplify the notation we are writing $\|\nabla u\|_{p,\Om,\nu}$
instead of  $\|\,|\nabla u|\,\|_{p,\Om,\nu}$).

\paragraph{\it Distribution function and rearrangements.}\quad
For a measurable function  $f$ in $\Sigma$ we define its \emph{distribution function with respect to $\mu$} as
\begin{equation*}
  f_{*\mu}(\tau)=\mu\big(\{x\in \Sigma: |f(x)|>\tau\}\big), \quad \tau>0,
\end{equation*}
and its \emph{nonincreasing rearrangement with respect to $\mu$} as
\begin{equation*}
  f^*_\mu(t)=\inf\{\tau>0: f_{*\mu}(\tau)\le t\},\quad t>0.
\end{equation*}
The function $f^\bigstar_\mu$ defined as
\begin{equation} \label{SymmRearr}
f^\bigstar_\mu(x) = f^*_\mu(C_D |x|^D),\quad x\in\Rde,
\end{equation}
where
\begin{equation}\label{unit_ball_in_sigma_measure}
C_D = \mu(B_1\cap \Sigma),
\end{equation}
is the \emph{radial rearrangement} of $f$ \emph{with respect to} $\mu$. We use the notation
\begin{equation*}
  B_r=\{x\in\Rde: |x|<r\},\quad r>0.
\end{equation*}
Thanks to the $\al$-homogeneity of $w$, we have
\begin{equation*}
\mu(B_r\cap \Sigma) = C_D r^D, \quad r>0.
\end{equation*}
Note that the function $f^\bigstar_\mu$ is nonnegative and radially nonincreasing.
Though it is defined on the whole $\Rde$, it depends only on function values of $f$ in $\Sigma$.
Furthermore, the functions $f$ and $f^\bigstar_\mu$ are equimeasurable with respect to $\mu$, that is,
\begin{equation*}
  \mu\big(\{x\in \Sigma: |f(x)|>\tau\}\big)
  =
  \mu\big(\{x\in\Sigma: |f^\bigstar_\mu(x)|>\tau\}\big), \quad \tau>0.
\end{equation*}

\subsection*{P\'olya-Szeg\"o inequality}

We need the following  important result.

\begin{thm}[{\cite[Theorem 3.1]{GLM23}}]\label{thm:PS}
Let $1\leq p < \infty$. For every function $u\in\mathcal C_c^1(\Rde)$,
its radial rearrangement $u^\bigstar_\mu$ with respect
to $\mu$ is locally absolutely continuous
and
\begin{equation}\label{thm:PS:ineq}
\|\nabla u^\bigstar_\mu\|_{p,\mu} \leq \|\nabla u\|_{p,\mu}.
\end{equation}
In particular,
$\spt u^\bigstar_\mu = \overline{B}_R$ where $R$
is such that $\mu(B_R \cap \Sigma) = \mu(\spt u \cap \Sigma)$.
\end{thm}

We need another obvious result.

\begin{lem}\label{RearrReduc}
Let $\Psi: (0,\infty)\to(0,\infty)$ be an increasing and measurable
function. Then
\begin{equation*}
  \int_{\Rde}\Psi\big(u_\mu^\bigstar(x)\big)\dd\mu(x)
  =
  \int_{\Rde}\Psi\big(|u(x)|\big)\dd\mu(x).
\end{equation*}
\end{lem}

\subsection*{Reduction to the one-dimensional problem}

\begin{prop}\label{PropAAA}
Let $u\in \mathcal{C}_c^1(\Rde)$ be a nonnegative function with
$\spt u\subset\Om$ such that $\mu(\Om)<\infty$.
Assume that $0<a\le c_D$ and
$R>0$ is the number satisfying $\mu(B_R \cap \Sigma) = \mu(\Om \cap \Sigma)$.
Let~$u_\mu^\bigstar$ be the
radial rearrangement of $u$ with respect to $\mu$.
Set
\begin{equation}\label{fuU}
  U(|x|)=u^\bigstar(x),\quad x\in\Rde\cap\Sigma
\end{equation}
and
\begin{equation}\label{ustar03}
  \f(t)=C_D^{1/D'}\,U(Re^{-t/D}),\quad t\in(0,\infty).
\end{equation}
Then $\f$
satisfies that  $\f'\ge0$ on $(0,\infty)$
and  $\f(0)=0$.
Moreover,
\begin{equation*}
  \int_{\Om\cap\Sigma}|\nabla u(x)|^D\dd\mu(x)
  \ge
  \int_{B_R\cap\Sigma}|\nabla u^\bigstar_\mu(x)|^D\dd\mu(x)
  =
  \int_{0}^{\infty}\big(\f'(t)\big)^D\dd t
\end{equation*}
and
\begin{multline*}
  \frac{1}{\mu(\Om\cap\Sigma)}\int_{\Om\cap\Sigma}\exp\big(a(u(x))^{D'}\big)\dd\mu(x) \\
  =
  \frac{1}{\mu(B_R\cap\Sigma)}\int_{B_R\cap\Sigma}\exp\big(a(u^\bigstar_\mu(x))^{D'}\big)\dd\mu(x)
  =
  \int_{0}^{\infty}\exp\big(\be \f^{D'}(t)-t\big)\dd t,
\end{multline*}
where $\be=a/C_D\in(0,1]$.
\end{prop}

\begin{proof}
By Lemma~\ref{RearrReduc} with $\Psi(t)=\exp(a t^{D'})$ we have
\begin{equation}\label{ustar01}
  \frac{1}{\mu(\Om\cap\Sigma)}
  \int_{\Om\cap\Sigma}\exp\big(a|u(x)|^{D'}\big)\dd\mu(x)
  =
  \frac{1}{\mu\big(B_R\cap\Sigma\big)}
  \int_{B_R\cap\Sigma}\exp\big(a (u^\bigstar(x))^{D'}\big)\dd\mu(x).
\end{equation}
Since $u^\bigstar$ is radially symmetric, we obtain
\begin{multline*}
  \int_{B_R\cap\Sigma}\exp\big(a (u^\bigstar(x))^{D'}\big)\dd\mu(x)
  =
  \int_{0}^{R}\Big(\int_{x\in\Sigma, |x|=r}\exp\big(a(u^\bigstar(x))^{D'}\big)w(x)\dd\sg(x)\Big)\dd r \\
  =
  \int_{0}^{R}\exp\big(a(U(r))^{D'}\big)\Big(\int_{x\in\Sigma, |x|=r}w(x)\dd\sg(x)\Big)\dd r
  =
  P_w(B_1\cap\Sigma)\int_{0}^{R}\exp\big(a(U(r))^{D'}\big)\,r^{D-1}\dd r,
\end{multline*}
where $U(|x|)=u^\bigstar(x)$ and $ P_w(B_1\cap\Sigma)=\int_{x\in\Sigma, |x|=1}w(x)\dd\sg(x)$
($\sg$ is the Hausdorff measure on the unit sphere centered at the origin).
We make the change of  variables
\begin{equation} \label{ChgVAR}
  r=Re^{-t/D}
\end{equation}
in the last integral to obtain
\begin{equation*}
  \int_{0}^{R}\exp\big(a(U(r))^{D'}\big)\,r^{D-1}\dd r
  =
  D^{-1}R^D\int_{0}^{\infty}\exp\big(a(U(Re^{-t/D}))^{D'}-t\big)\dd t.
\end{equation*}
Setting $\be=a/C_D\in(0,1]$ and using the fact, that
\begin{equation*}
  C_D
  =
  \mu\big(B_R\cap\Sigma\big)
  =\int_{B_R\cap\Sigma}w(x)\dd x=D^{-1} P_w(B_1\cap\Sigma)R^D,
\end{equation*}
we arrive at
\begin{equation} \label{ustar02}
  \frac{1}{\mu\big(B_R\cap\Sigma\big)}
  \int_{B_R\cap\Sigma}\exp\big(a(u^\bigstar(x))^{D'}\big)\dd\mu(x)
  =
  \int_{0}^{\infty}\exp\big(\be(\f(t))^{D'}-t\big)\dd t,
\end{equation}
where the function $\f$ is defined by \eqref{ustar03}.
Similarly, since $\big|\nabla u^\bigstar(x)\big|=\big|U'\big(|x|\big)\big|$,
$x\in\Rde\cap\Sigma$, we have
\begin{equation*}
  \int_{B_R\cap\Sigma}\big|\nabla u^\bigstar(x)\big|^D\dd\mu(x)
  =
  P_w(B_1\cap\Sigma)\int_{0}^{R}\big|U'(r)\big|^D\,r^{D-1}\dd r.
\end{equation*}
Differentiating~\eqref{ustar03} we have, for all $t\in(0,\infty)$,
\begin{equation*}
  \f'(t)=-D^{-1}C_D^{1/D'}\,U'(Re^{-t/D})Re^{-t/D}
  =
  -D^{-1}C_D^{1/D'}\,U'(r)r.
\end{equation*}
Observe that, since $U$ is nonincreasing, the right hand side is nonnegative.
Change of variables \eqref{ChgVAR}  gives
\begin{multline*}
  \int_{0}^{R}\big|U'(r)\big|^D\,r^{D-1}\dd r
  =
  D^{-1}(D^{-1}C_D^{1/D'})^{-D}\,\int_{0}^{\infty}\big|D^{-1}C_D^{1/D'}
  \,U'(Re^{-t/D})Re^{-t/D}\big|^D\dd t \\
  =
  D^{-1}(D^{-1}C_D^{1/D'})^{-D}\,\int_{0}^{\infty}|\f'(t)|^D\dd t.
\end{multline*}
That is,
\begin{equation*}
   \int_{B_R\cap\Sigma}\big|\nabla u^\bigstar(x)\big|^D\dd\mu(x)
   =
   P_w(B_1\cap\Sigma)D^{-1}(D^{-1}C_D^{1/D'})^{-D}\,\int_{0}^{\infty}|\f'(t)|^D\dd t
   =
   \int_{0}^{\infty}|\f'(t)|^D\dd t.
 \end{equation*}
 The proof is completed.
 \end{proof}

\begin{lem}\label{L01}
Let $q\in(1,\infty)$. Consider nonnegative
functions $\f\in C^1[0,\infty)$ such that $\f(0)=0$, $\f'\ge0$ on $(0,\infty)$
and $\int_{0}^{\infty}(\f'(t))^q\dd t\le1$. Then
\begin{equation}\label{MoserSup}
  A(q)=\sup_{\f}\int_{0}^{\infty}\exp\big((\f(t))^{q'}-t\big)\dd t<\infty.
\end{equation}
\end{lem}

\begin{proof}
The proof for $q\ge2$ was given by Moser~\cite{Moser71},
for $q\in(1,2)$ it was extended by Jodeit~\cite{Jod}.
\end{proof}

\begin{lem}\label{L02}
The supremum in \eqref{MoserSup} is attained for every $q>1$.
\end{lem}

\begin{proof}
The result for $q\in\Ne$, $q\ge2$ was  proved first by
by Carleson and Chang~\cite{CaCh}, it was extended to real $q>1$ by Hudson and Leckband~\cite{HuLe02}.
\end{proof}

\section{Main result}

Now we formulate the main theorem.

\begin{thm}
Assume that $p\in[1,\infty)$ and $w$ is an $\al$-homogeneous weight such that $\al>0$.
Then
\begin{equation*}
  S=\sup_u\frac{1}{\mu(\Om\cap\Sigma)}
  \int_{\Om\cap\Sigma}\exp\big(a|u(x)|^{D'}\big)\dd\mu(x) <\infty,
\end{equation*}
where the supremum is taken over all $u\in \mathcal{C}_c^1(\Rde)$ such that
$\spt u\subset\Om$ with $\mu(\Om)<\infty$ and $\|\nabla u\|_{p,\mu}\le1$.
If $\Om=B_R$ with some $R>0$, then the supremum is attained.
\end{thm}

\begin{proof}
To prove that $S$ is finite we apply Proposition~\ref{PropAAA} to obtain
\begin{equation*}
  1\ge\int_{\Om\cap\Sigma}|\nabla u(x)|^D\dd\mu(x)
  \ge
  \int_{0}^{\infty}\big(\f'(t)\big)^D\dd t,
\end{equation*}
and so, by Lemma~\ref{L01} and Proposition~\ref{PropAAA},
\begin{equation*}
  \sup_u\frac{1}{\mu(\Om\cap\Sigma)}
  \int_{\Om\cap\Sigma}\exp\big(a|u(x)|^{D'}\big)\dd\mu(x)
  =
  \sup_\f\int_{0}^{\infty}\exp\big((\f(t))^{D'}-t\big)\dd t<\infty.
\end{equation*}
For proving the second part it is enough to take the function
\begin{equation*}
  u(x)=U(|x|), \quad x\in B_R\cap\Sigma,
\end{equation*}
where $U$ is the function such that \eqref{ustar03} holds
and $\f$ is the extremal function from Lemma~\ref{L02} when $q=D$.
\end{proof}


\end{document}